\documentclass[titlepage]{amsart}

\usepackage[foot]{amsaddr}

\usepackage{graphicx}
\usepackage{amsmath,amsthm}
\usepackage{amssymb}
\usepackage{mathrsfs}
\usepackage{pinlabel}
\usepackage[left = 3cm, right = 3 cm , top = 3 cm , bottom = 3cm ]{geometry}
\usepackage[utf8]{inputenc}
\usepackage{cite}
\usepackage{todonotes}
\usepackage{enumitem}
\usepackage{appendix}

\newtheoremstyle{theorem}
                 {5pt}
                 {5pt}
                 {}
                 {}
                 {\bfseries}
                 {\newline}
                 {0pt}
                 {}

\theoremstyle{theorem}
\newtheorem{theorem}{Theorem}[section]
\newtheorem{lemma}[theorem]{Lemma}
\newtheorem{definition}[theorem]{Definition}

\numberwithin{theorem}{section}
\numberwithin{equation}{section}

\title{Curvature Conditions for Spatial Isotropy}

\author[Kostas Tzannavaris]{Kostas Tzanavaris}
\author[Pau Amaro Seoane]{Pau Amaro Seoane}
\address[Kostas Tzannavaris]{Higgs Centre for Theoretical Physics, School of Physics and Astronomy, University of Edinburgh, Scotland, UK.}
\email[Kostas Tzannavaris]{kostas.tzan@ed.ac.uk}

\address[Pau Amaro Seoane]{Institute for Multidisciplinary Mathematics, Universitat Politècnica de València, Spain\\
Max Planck Institute for Extraterrestrial Physics\\
Garching, Germany\\
Institute of Applied Mathematics, Academy of Mathematics and Systems Science, CAS, Beijing\\
China, Kavli Institute for Astronomy and Astrophysics, Beijing, China.}
\email[Pau Amaro Seoane]{amaro@riseup.net}

\begin{document}
\begin{abstract}
In the context of mathematical cosmology, the study of necessary and sufficient conditions for a semi-Riemannian manifold to be a (generalised) Robertson-Walker space-time is important. In particular, it is a requirement for the development of initial data to reproduce or approximate the standard cosmological model. Usually these conditions involve the Einstein field equations, which change if one considers alternative theories of gravity or if the coupling matter fields change. Therefore, the derivation of conditions which do not depend on the field equations is an advantage. In this work we present a geometric derivation of such a condition. We require the existence of a unit vector field to distinguish at each point of space two (non-equal) sectional curvatures. This is equivalent for the Riemann tensor to adopt a specific form. Our geometrical approach yields a local isometry between the space and a Robertson-Walker space of the same dimension, curvature and metric tensor sign (the dimension of the largest subspace on which the metric tensor is negative definite). Remarkably, if the space is simply-connected, the isometry is global. Our result generalize to a class of spaces of non-constant curvature the theorem that spaces of the same constant curvature, dimension and metric tensor sign must be locally isometric. Because we do not make any assumptions regarding field equations, matter fields or metric tensor sign, one can readily use this result to study cosmological models within alternative theories of gravity or with different matter fields.
\end{abstract}

\maketitle

\section{Introduction}
\label{sec.1}

In a paper published in 2014, Chen \cite{2014GReGr..46.1833C} showed that if a Lorentz manifold $(M,\,g)$ possesses a timelike vector field $X$ for which there exists a real-valued function $f:M\rightarrow\mathbf{R}$ such that the following condition is fulfilled,
\begin{equation}
    \label{eqn:1-1}
    X_{\alpha ; \,\beta} = f g_{\alpha\beta},
\end{equation}

\noindent
then $M$ can be expressed as the (warped) product $I\times\Sigma$ of an interval $I$ and a spacelike hypersurface $(\Sigma,\,\sigma)$, with the metric $g$ having the form
\begin{equation}
    \label{eqn:1-2}
    g= -\,dt\otimes dt + a^2(t) \sigma.
\end{equation}

\noindent
Here $a$ is a real-valued function on $M$ which is constant on the hypersurfaces $t=\textrm{constant}$. Spaces of this form are usually referred to as \textit{Generalized Robertson-Walker spaces} (GRW). If the curvature of $(\Sigma,\,\sigma)$ is constant, then it is a \textit{Robertson-Walker space} (RW). 

\vspace{0.5em}

Chen's result has been used in subsequent works to derive necessary and sufficient conditions for a space-time to be a GRW or RW space. In particular, the research of \cite{2016JMP....57j2502M,2019JMP....60e2506M,2019IJGMM..1650016D} focuses on the characterisation of GRW and RW spaces via conditions on the curvature tensor, more specifically the Ricci and Weyl ones.

\vspace{0.5em}

In this work we follow a different approach. Instead of assuming the existence of a timelike cocircular vector field, i.e. a field that satisfies~Eq.\eqref{eqn:1-1}, we require the existence of a unit vector field $u$ such that the Riemann tensor is
\begin{align}
    \label{eqn:1-3}
    & R(X,\,u)\,u = fX, \\[5pt]
    \label{eqn:1-4}
    & R(X,\,Y)Z = h\left(g(Y,\,Z)X - g(X,\,Z)Y\right),
\end{align}

\noindent
where $f,\,h$ are real-valued functions and $X,\,Y,\,Z$ are \textit{arbitrary} vector fields perpendicular to $u$. \textit{The main result of our work is proving that this is a necessary and sufficient criterion for the space-time to be a RW space.}

\vspace{0.5em}

The condition above is equivalent to requiring that the vector field $u$ distinguishes at each point $p$ two sectional curvatures in the following way: Let $\Pi_1,\,\Pi_2$ be arbitrary planes on the tangent space $T_p\,M$. If both planes are either perpendicular to $u$ or contain $u$, then their sectional curvatures must be equal. 

\vspace{0.5em}

We note that the curvature of the space is constant if the sectional curvatures of the two arbitrary planes $\Pi_1$ and $\Pi_2$ (with $\Pi_1$ being perpendicular to $u$ and $\Pi_2$ containing $u$) are equal. This scenario corresponds to the well-known result that if two spaces have the same
dimension, metric sign and constant curvature, then they are locally isometric (see e.g. \cite{o1983semi,wolf2011curvature}) and therefore we will not consider this situation in our work. This condition is expressed using the functions $f$ and $h$, as defined by~Eqs.\eqref{eqn:1-3} and \eqref{eqn:1-4} as follows

\begin{equation}
    \label{eqn:1-5}
    h(p)-\varepsilon f(p) \neq 0,\textrm{ for every }p\in M
\end{equation}

\noindent
where $\varepsilon\equiv g(u,u) = \pm 1$. 

\vspace{0.5em}

In Section ~(\ref{sec.Statement}) we introduce the theorem and sketch its proof. The full proof is presented in Section ~(\ref{sec.Proof}), with most of the calculations being done in Sections ~(\ref{Sec.Riem}) and ~(\ref{Sec.Bianchi}), where we determine the Riemann tensor of the spaces under study and subsequently derive some important formulae as a consequence of the Bianchi identities. Lastly, in Sec.~(\ref{sec.Impl}) we discuss the implications of our findings.

\section{Statement of the theorem}
\label{sec.Statement}

\noindent
Before stating the theorem, we will briefly state some definitions.

\begin{definition}
A semi-Riemannian space $(M,\,g)$ is said to be 

\begin{enumerate}

	\item \textit{locally isotropic} with respect to a unit vector field $u$ if and only if the following condition is satisfied: Let $p$ be an arbitrary point of $M$ and $\Lambda$ be a linear isometry of $T_p\, M$ that leaves $u\rvert_p$ invariant. There exists a local isometry $\phi:\mathscr{U}\rightarrow M$ defined on an open 	neighborhood $\mathscr{U}$ of $p$ such that $\phi(p)=p$ and $\phi_*\rvert_p = \Lambda$.

	\item a \textit{generalized Robertson-Walker space} (GRW) with sign $\varepsilon = \pm 1$ if and only if it is the warped product of an open interval $I$ and a semi-Riemannian space $(\Sigma,\sigma)$ such that
	
    \begin{equation}
        \label{eqn:2-1}
    	g = \varepsilon dt\otimes dt + a^2(t)\sigma,
    \end{equation}

    \noindent
    where $a$ is a real-valued function on $M$ which is constant on the hypersurfaces $t=\textrm{constant}$.

	\item a \textit{Robertson-Walker space} (RW) with sign $\varepsilon = \pm 1$ if and only if it is a GRW space, with $(\Sigma,\sigma)$ having constant curvature.
\end{enumerate}

\end{definition}

GRW and RW spaces are usually defined as Lorentz spaces, with the vector field $u=\partial_t$ being time-like. However, as we will prove, our theorem is satisfied for general semi-Riemannian spaces. 

\begin{theorem}
\label{thm-main}

Let $(M,g)$ be a semi-Riemannian space of dimension $n \geq 4$ and of non-constant curvature, and $u$ be a unit vector field. The following propositions are equivalent:

\begin{enumerate}
	\item $(M,g)$ is locally isotropic with respect to $u$.
	\item The Riemann tensor of $(M,g)$ is defined by \eqref{eqn:1-3} and \eqref{eqn:1-4}.
	\item $(M,g)$ is locally isometric to a RW space such that $u=\partial_t$.
\end{enumerate}

In addition, if $M$ is simply-connected, then $M$ is isometric to a RW space.

\end{theorem}

The proof follows a derivation of the RW metric given by \cite{kriele1999spacetime}, in which the author uses the Einstein field equations with matter described as a perfect fluid.  The proof that $(1)$ implies $(2)$ is presented in the following section. The proof that $(3)$ implies $(1)$ relies on the local isometry of a space of constant curvature to one of the simply-connected space forms $\mathbf{R}_{\,\nu}^{n-1},\,\mathbf{S}_\nu^{n-1}(R),\,\mathbf{H}_\nu^{n-1}(R)$ (see \cite{o1983semi} for the definition of these spaces). It therefore remains to prove that $(2)$ implies $(3)$.

\vspace{0.5em}

The first (and most demanding) step, is to show that $u$ is locally orthogonal to a one-parameter family of hypersurfaces of $M$. This follows from an integrability condition via Poincar\'e's lemma, from which it also follows that the hypersurfaces can be extended globally if $M$ is simply-connected. This will be proved by a set of conditions (involving $f$ and $h$) derived from the second Bianchi identity. The second step is to show that every hypersurface has a constant curvature, which can be derived from Gauss formula for curvatures of hypersurfaces. The third and last step is to prove that for every vector field $X$ orthogonal to $u$, $df(X)=0$ and $dh(X)=0$. Once we have proved these three steps, the proof of the theorem is straightforward, as we will see.


\section{The Riemann tensor of an isotropic space}
\label{Sec.Riem}

Throughout this section, we will denote by $R$ the Riemann tensor of a semi-Riemannian manifold $(M,g)$ that is locally isotropic with respect to a unit vector field $u$, and whose dimension is $\geq 4$.

\begin{lemma}

\label{lem.RXUY}
Let $x,y\in u\rvert^\perp_p$ two unit vectors. Then

\begin{equation}
    R(x,u)y = -R(y,u)x.\label{eqn:A-1}
\end{equation}

\noindent 
Additionally, if $x,y$ are of the same causal character then

\begin{equation}
    R(x,u)x = R(y,u)y,
\end{equation}

\noindent 
while if $x,y$ are of different causal character then

\begin{equation}
    R(x,u)x = -R(y,u)y.\label{eqn:A-3}
\end{equation}
\end{lemma}

\begin{proof}

First, suppose that the vectors x,y are of the same causal character. Consider
the isometry $\ell$ on the plane spanned by $x,y$ defined by the relation
\begin{align}
    & \ell x = \cos\theta x - \sin\theta y,\label{eqn:A-4}\\
    & \ell y = \sin\theta x + \cos\theta y,\label{eqn:A-5}
\end{align}

\noindent 
leaving invariant any vector $z$ which is orthogonal to that plane:
\begin{equation}
    \ell z = z, \;\; z\perp x,y.\label{eqn:A-6}
\end{equation}

\noindent 
Because we are working with an isotropic space, we have
\begin{align}
g\big(R(x,u)y,z\big)	&=	g\big(R(\ell x,u)\ell y,z\big) \label{eqn:A-7}\\
						&=	\sin\theta g\big(R(\ell x,u)x,z\big) + \cos\theta g\big(R(\ell x,u)y,z\big)\nonumber\\
						&=	\begin{aligned}[t]
								& \sin\theta\cos\theta g\big(R(x,u)x,z\big) - \sin^2\theta g\big(R(y,u)x,z\big)\\
								& + \cos\theta^2 g\big(R(x,u)y,z\big)-\sin\theta\cos\theta g\big(R(y,u)y,z\big)
							\end{aligned}\nonumber\\
						&=	\begin{aligned}[t]
								& - \sin^2\theta \Big[g\big(R(x,u)y,z\big) + g\big(R(y,u)x,z\big)\Big]\\
								& + \sin\theta\cos\theta \Big[g\big(R(x,u)x,z\big) - g\big(R(y,u)y,z\big)\Big]\\
								& + g\big(R(x,u)\,y,z\big).
							\end{aligned}\nonumber
\end{align}

\noindent 
Because this equation holds for every real number $\theta$, we have that
\begin{align}
    & g\big(R(x,u)y + R(y,u)x,z\big) = 0,\label{eqn:A-8}\\
    & g\big(R(x,u)x - R(y,u)y,z\big) = 0.\label{eqn:A-9}
\end{align}

\noindent 
We now derive the components parallel to the vectors $x,y$. Since
\begin{equation}
    R(\ell x,\ell y) = R(\cos\theta x - \sin\theta y, \sin\theta x + \cos\theta y) = R(x,y),
\label{eqn:A-10}
\end{equation}

\noindent 
it follows that
\begin{align}
g\big(R(x,u)y,x\big) 	&= g\big(R(\ell y,\ell x)\ell x,u\big)  \label{eqn:A-11}\\
                        &= g\big(R(y,x)\ell x,u\big)\nonumber\\
						&= \cos\theta g\big(R(y,u)x,u\big) - \sin\theta g\big(R(y,x)y,u\big)\nonumber\\
						&= \cos\theta g\big(R(x,u)y,x\big) - \sin\theta g\big(R(y,u)y,x\big).\nonumber
\end{align}

\noindent 
For $\theta=\pi$, the last equation \eqref{eqn:A-11} yields
\begin{equation}
    g\big(R(x,u)y,x\big) = 0 = -g\big(R(y,u)x,x\big)
\label{eqn:A-12}
\end{equation}

\noindent 
and therefore
\begin{equation}
    g\big(R(y,u)y,x\big) = 0.\label{eqn:A-13}
\end{equation}

\noindent 
Interchanging the vectors $x$ and $y$ in \eqref{eqn:A-13} gives us
\begin{equation}
    g\big(R(x,u)x,y\big) = 0,
\label{eqn:A-14}
\end{equation}

\noindent 
proving the first half of the lemma.

\vspace{0.5em}

Suppose now that $x,y$ are of different causal character. We define the isometry $\Lambda$ on the 
plane spanned by $x,y$ which leaves any vector z orthogonal to that plane invariant as
\begin{align}
    & \Lambda x = \cosh\theta x + \sinh\theta y,\label{eqn:A-15}\\
    & \Lambda y = \sinh\theta x + \cosh\theta y,\label{eqn:A-16}\\
    & \Lambda z = z, \;\; z\perp x,y.\label{eqn:A-17}
\end{align}

\noindent 
Again, since we are dealing with isotropic spaces,
\begin{align}
g\big(R(x,u)y,z\big)	&=	g\big(R(\Lambda x,u)\Lambda y,z\big)\label{eqn:A-18}\\
						&=	\begin{aligned}[t]
								& g\big(R(x,u)y,z\big)\\
								& + \sinh^2\theta \Big[g\big(R(x,u)y,z\big) + g\big(R(y,u)x,z\big)\Big]\\
								& + \sinh\theta\cosh\theta\Big[g\big(R(x,u)x,z\big) + g\big(R(y,u)y,z\big)\Big].
							\end{aligned}\nonumber
\end{align}

\noindent 
This relation holds for every real number $\theta$ and thus
\begin{align}
    g\big(R(x,u)y + R(y,u)x,z\big) = 0,\label{eqn:A-19}\\
    g\big(R(x,u)x + R(y,u)y,z\big) = 0,\label{eqn:A-20}
\end{align}

\noindent 
The isometry $\Lambda$ also preserves the Riemann operator
\begin{equation}
    R(\Lambda x,\Lambda y) = R(\cosh\theta x + \sinh\theta y, \sinh\theta x + \cosh\theta y) = R(x,y),
\label{eqn:A-21}
\end{equation}

\noindent 
from which it follows that
\begin{align}
g\big(R(x,u)y,x)\big) 	&= g\big(R(\Lambda y,\Lambda x)\Lambda x,u\big)\nonumber\\
						&= g\big(R(y,x)\Lambda x,u\big)\nonumber\\
						&= \cosh\theta g\big(R(y,x)x,u\big) + \sinh\theta g\big(R(y,x)y,u\big)\nonumber\\
						&= \cosh\theta g\big(R(x,u)y,x\big) + \sinh\theta g\big(R(y,u)y,x\big).\label{eqn:A-22}
\end{align}

\noindent 
As such
\begin{equation}
    g\big(R(x,u)y,x\big) = -g\big(R(y,u)x,x\big) = 0,
\label{eqn:A-23}
\end{equation}

\noindent 
and as a consequence
\begin{equation}
    g\big(R(x,u)x,y\big) = g\big(R(y,u)y,x\big) = 0,
\label{eqn:A-24}
\end{equation}

\noindent 
which completes the proof.
\end{proof}

\begin{lemma}
\label{lem.XYZOrth}

\noindent 
If $X,Y,Z$ vector fields orthogonal to $u$, then the vector field $R(X,Y)Z$ is
also orthogonal to $u$.

\end{lemma}

\begin{proof}
The proof is basically a combinatorial argument involving the first Bianchi
identity. Without loss in generality, we can suppose that $X,Y,Z$ are
orthogonal to each other. For some point $p\in M$ let $x=X\rvert_p$,
$x=Y\rvert_p$ and $z=Z\rvert_p$. We will first show that the mapping

\begin{equation}
R_u(x,y,z)\equiv g\big(R(x,y)z,u\big)\label{eqn:A-25}
\end{equation}

\noindent 
is antisymmetric. Because of the antisymmetry property, we have that 
$R_u(x,y,z) = -R_u(y,x,z)$. It also holds that
\begin{align}
R_u(x,y,z) 	&= g\big(R(z,u)x,y\big) \label{eqn:A-26}\\
			&= -g\big(R(x,u)z,u\big) \nonumber\\
			&= -g\big(R(z,y)x,u\big) \nonumber\\
			&= -R_u(z,y,z)\nonumber
\end{align}

\noindent 
and
\begin{align}
R_u(x,y,y)	&= g\big(R(x,y)y,u\big) \label{eqn:A-27}\\
			&= g\big(R(x,u)x,y\big) \nonumber\\
			&= -g\big(R(x,u)y,y\big) \nonumber\\
			&=0,\nonumber
\end{align}

\noindent 
which means that
\begin{equation}
    0 = R_u(x,y+z,y+z) = R_u(x,y,z) + R_u(x,z,y),
\label{eqn:A-28}
\end{equation}

\noindent 
showing that $R_u$ is antisymmetric. Finally, from the first Bianchi identity
it follows that
\begin{equation}
    0=R_u(x,y,z) + R_u(y,z,x) + R_u(z,x,y) = 3R_u(x,y,z).
\label{eqn:A-29}
\end{equation}

\end{proof}

\begin{lemma}
\label{lem.PiXY}

Let $\Pi_{x,y}$ be the non-degenerate plane of $T_pM$ spanned by two vectors $x,y$.

\begin{enumerate}
	\item For every pair of unit vectors $x,y$ orthogonal to $u\rvert_p$ and to each other,
	the sectional curvatures $K_p(\Pi_{x,u}),K_p(\Pi_{y,u})$ of the planes $\Pi_{x,u},\Pi_{y,u}$ are equal:
	\begin{equation}
		K_p(\Pi_{x,u}) = K_p(\Pi_{y,u}).\label{eqn:A-30}
	\end{equation}
	
	\item For every three vectors $x,y,z$ orthogonal to $u\rvert_p$ and to each other, the sectional
	curvatures $K_p(\Pi_{x,y}),K_p(\Pi_{y,z})$ of the planes $\Pi_{x,y},\Pi_{y,z}$ are equal:
	\begin{equation}
		K_p(\Pi_{x,y}) = K_p(\Pi_{y,z}).\label{eqn:A-31}
	\end{equation}
\end{enumerate}
\end{lemma}

\begin{proof}

The first part is a direct consequence of Lemma \ref{lem.RXUY}, so we
need to prove only the second part. In the case in which all three vectors have
the same causal character this follows by considering a rotation on the plane
spanned by $x$ and $y$. We will therefore prove only the case in which one of
the vectors -- say $z$ -- is of different causal character than $x,y$. Let us
consider the isometry $\Lambda$ defined by
\begin{align}
    & \Lambda x = \cosh\theta x + \sinh\theta z,\label{eqn:A-32}\\
    & \Lambda z = \sinh\theta x + \cosh\theta z,\label{eqn:A-33}\\
    & \Lambda w = w, \;\; w\perp x,z. \label{eqn:A-34}
\end{align}

\noindent 
Then

\begin{align}
g\big(R(x,y)y,x\big)	&= g\big(R(\Lambda x,y)y,\Lambda x\big) = \label{eqn:A-35}\\
						&= \cosh\theta g\big(R(x,y)y,\Lambda x\big) + \sinh\theta g\big(R(z,y)y,\Lambda x\big)\nonumber\\
						&=	\begin{aligned}[t]
								& \cosh^2\theta g\big(R(x,y)y,x\big) + \cosh\theta \sinh\theta g\big(R(x,y)y,z\big)\\
								& + \cosh\theta \sinh\theta g\big(R(z,y)y,x\big) + \sinh^2\theta g\big(R(z,y)y,z\big)
							\end{aligned}\nonumber\\
						&=	\begin{aligned}[t]
								& \sinh^2\theta \Big[g\big(R(x,y)y,x\big) + g\big(R(z,y)y,z\big)\Big]\\
								& + \cosh\theta \sinh\theta \Big[g\big(R(z,y)y,x\big) + g\big(R(x,y)y,z\big)\Big]\\
								& + g\big(R(x,y)y,x\big).
							\end{aligned}\nonumber
\end{align}

\noindent 
Since this holds for every real number $\theta$,

\begin{equation}
    g\big(R(x,y)y,x\big) = -g\big(R(y,x)x,y\big), 
\label{eqn:A-36}
\end{equation}

\noindent 
which proves the lemma.
\end{proof}

\noindent
We are now in a position to prove that the condition of local isotropy completely specifies the Riemann tensor.

\begin{theorem}
If $(M,g)$ is a semi-Riemannian manifold with dimension $\geq 4$ that is locally isotropic with respect to a unit vector $u$, its Riemann tensor satisfies Eqs.~\eqref{eqn:1-3}, \eqref{eqn:1-4}. 
\end{theorem}

\begin{proof}
We will first prove \eqref{eqn:1-3}. At any point $p$ of $M$ the sectional curvature $\Psi(p)$ of a non-degenerate plane containing $u\vert_p$ is independent of the choice of the plane, and therefore it is a pointwise function on $M$.  Consider now two vectors $x,y$ orthogonal to $u\rvert_p$ and a real number $\lambda$ such that the vector $x+\lambda y$ is not null. Then Lemma \ref{lem.PiXY} implies that:
\begin{equation}
    g\big(R(x+\lambda y,u)u,x+\lambda y\big) = \varepsilon g(x+\lambda y,x+\lambda y)\Psi(p).\label{eqn:A-37}
\end{equation}

\noindent 
Thanks to continuity, \eqref{eqn:A-37} holds for $\lambda=1$, which leads us to

\begin{equation}
    g\big(R(x,u)u - \varepsilon\Psi(p)x,y\big) = 0 \label{eqn:A-38}
\end{equation}

\noindent 
for every vector $y$ orthogonal to $u$, which proves this for $f=\varepsilon\Psi$.

\vspace{0.5em}

We will now prove \eqref{eqn:1-4}. Using Lemma \ref{lem.PiXY} along with the algebraic properties of the Riemann tensor 
(cf. \cite{kobayashi1963geom1}) it follows that at each point $p$ there exists a real 
number $h(p)$ such that

\begin{equation}
    g\Big(R(x,y)z -h(p)\big(g(y,z)x - g(x,z)y\big),w\Big) = 0
\label{eqn:A-39}
\end{equation}

\noindent 
for every $x,y,z,w\in T_pM$ orthogonal to $u\vert_p$. Additionally, using Lemma
\ref{lem.XYZOrth} we get

\begin{equation}
    g\Big(R(x,y)z -h(p)\big(g(y,z)x - g(x,z)y\big),u\Big) = 0.
\label{eqn:A-40}
\end{equation}
\end{proof}

\begin{theorem}
If the Riemann tensor of a semi-Riemannian manifold $(M,g)$ satisfies the relations
\begin{align}
    & R(X,u)u = fX, \label{eqn:A-41}\\[5pt]
    & R(X,Y)Z = h\big(g(Y,Z)X - g(X,Z)Y\big),\label{eqn:A-42}
\end{align}

\noindent 
with $u$ a unit vector field and $X,Y,Z$ vector fields orthogonal to $u$, then it also 
must satisfy the following relations:
\begin{align}
    & R(X,Y)u = 0,                      \label{eqn:A-43}\\
    & R(X,u)Y = -\varepsilon f g(X,Y)u. \label{eqn:A-44}
\end{align}
\end{theorem}

\begin{proof}
In order to prove the first relation, we will prove that $R(X,Y)u$ is orthogonal to every vector field. Eq. \eqref{eqn:A-41} implies that
\begin{equation}
    g\big(R(X,Y)u,u\big) = -g\big(R(X,Y)u,u\big) = 0.
\end{equation}
Now, if $Z$ is an arbitrary vector field orthogonal to $u$ then it's also orthogonal to $R(X,Y)u$ due to \eqref{eqn:A-42}:
\begin{equation}
    g\big(R(X,Y)u,Z\big) = - g\big(R(X,Y)Z,u\big) = 0.
\end{equation}
Eq. \eqref{eqn:A-44} is proved similarly:
\begin{equation}
    g\big(R(X,u)Y,u\big) = -g\big(R(X,u)u,Y\big) = -fg(X,Y).
\end{equation}
\end{proof}

\noindent
Since every other component of the Riemann tensor vanishes identically, it follows that the conditions \eqref{eqn:1-3}, \eqref{eqn:1-4} completely specify the Riemann tensor. 

\section{The Bianchi identities}
\label{Sec.Bianchi}

We will now use the second Bianchi identities to derive some important relations involving the differentials $df,dh$ of the sectional curvatures $f,h$ respectively. The main result proved in this section is the following. 

\begin{theorem}
\label{lem.OrthCond}
Let $(M,g)$ be a semi-Riemannian manifold with dimension $\geq 4$ whose Riemann tensor $R$ satisfies \eqref{eqn:1-3} and \eqref{eqn:1-4}. If $X,Y$ are vector fields on $M$ that are orthogonal to $u$, then
\begin{align}
& df(X) = -(h-\varepsilon f) g(X,\nabla_u u), \label{eqn:3-1} \\
& g(\nabla_X u, Y) = g(\nabla_Y u, X), \label{eqn:3-2} \\
& dh(X) = 0. \label{eqn:3-3}
\end{align}
\end{theorem}

\begin{proof}

Instead of working with arbitrary vector fields, we shall use vector 
fields $X,Y,Z$ orthogonal to $u$ that are Fermi-parallel along an 
integral curve of $u$ and therefore satisfy the following equations:
\begin{eqnarray}
&& \nabla_u X = -\varepsilon g\left(X,\nabla_u u\right) u, \label{eqn:C-1}\\
&& \nabla_u Y = -\varepsilon g\left(Y,\nabla_u u\right) u, \label{eqn:C-2}\\
&& \nabla_u Z = -\varepsilon g\left(Z,\nabla_u u\right) u. \label{eqn:C-3}
\end{eqnarray}
Given that 
\begin{enumerate}
\item the relations to be proved are tensor equations,
\item the dimension of the space-time is at least $4$,
\item a Fermi-parallel basis of vector fields always exist,
\end{enumerate}
\noindent 
if Theorem \ref{lem.OrthCond} holds for every pair of vector fields orthogonal
to $u$ and Fermi-parallel along an arbitrary integral curve of u then
it holds for every pair of vector fields orthogonal to $u$.

\vspace{0.5em}

\noindent
To see this we begin by considering the second Bianchi identity:
\begin{equation}
(\nabla_u R)(X,Y)Z + (\nabla_X R)(Y,u)Z + (\nabla_Y R)(u,X)Z = 0.
\label{eqn:C-4}
\end{equation}

\noindent
We have
\begin{align}
(\nabla_u R)(X,Y)Z &=			\begin{aligned}[t]
									& \nabla_u\big(R(X,Y)Z\big) - R(\nabla_u X,Y)Z \\
									& - R(X,\nabla_u Y)Z - R(X,Y)\nabla_u Z
								\end{aligned}\nonumber\\
									&= \nabla_u\big(R(X,Y)Z\big) - R(\nabla_u X,Y)Z - R(X,\nabla_u Y)Z, \label{eqn:C-5}\\
\nonumber\\
\nabla_u\big(R(X,Y)Z\big) &= 	\begin{aligned}[t]
									& dh(u)\big(g(Y,Z)X - g(X,Z)Y\big) \\
									& + h\big(g(Y,Z)\nabla_u X - g(X,Z)\nabla_u Y \big),
								\end{aligned}
\label{eqn:C-6}\\
\nonumber\\
R(X,\nabla_u Y)Z &=	 -\varepsilon g(Y,\nabla_u u)R(X,u)Z = \varepsilon f g(X,Z)\nabla_u Y,
\label{eqn:C-7}\\
\nonumber\\
R(\nabla_u X,Y)Z &= -\varepsilon fg(Y,Z)\nabla_u X.
\label{eqn:C-8}\\
\nonumber
\end{align}

\noindent 
Replacing eqs. \eqref{eqn:C-6},\eqref{eqn:C-7} and \eqref{eqn:C-5} gives the first term in the second Bianchi identity \eqref{eqn:C-4}:

\begin{equation}
(\nabla_u R)(X,Y)Z = 	\begin{aligned}[t]
							& (h-\varepsilon f)\big( g(Y,Z)\nabla_u X - g(X,Z)\nabla_u Y \big) \\
							& + dh(u)\big( g(Y,Z)X - g(X,Z)Y \big).
						\end{aligned}\label{eqn:C-9}
\end{equation}

\noindent 
We now address the second term,

\begin{align}
(\nabla_X R)(Y,u)Z &=			\begin{aligned}[t]
									&\nabla_X\big(R(Y,u)Z\big) - R(\nabla_X Y,u)Z \\
									&-R(Y,\nabla_X u)Z - R(Y,u)\nabla_X Z,
								\end{aligned}
\label{eqn:C-10}\\
\nonumber\\
\nabla_X\big(R(Y,u)Z\big) 	&=	 \nabla_X\big(-\varepsilon fg(Y,Z)u\big)\label{eqn:C-11}\\
							&=	-\varepsilon df(X)g(Y,Z)u - \varepsilon fg(Y,Z)\nabla_X u,\nonumber\\
\nonumber\\
R(\nabla_X Y,u)Z &= R\big((\nabla_X Y)_\perp,u\big)Z = -\varepsilon f g(\nabla_X Y,Z)u.\label{eqn:C-12}
\end{align}

\noindent 
Since $u$ is of unit norm, the vector field $\nabla_X u$ is orthogonal to $u$, and then

\begin{equation}
R(Y,\nabla_X u)Z = h\big(g(\nabla_X u,Z)Y - g(Y,Z)\nabla_X u\big).\label{eqn:C-13}
\end{equation}

\noindent 
Lastly, we have

\begin{align}
\nabla_X Z &= (\nabla_X Z)_\perp + \varepsilon g(\nabla_X Z,u)u \label{eqn:C-14}\\
&= (\nabla_X Z)_\perp - \varepsilon g(Z, \nabla_X u)u \nonumber
\end{align}

\noindent 
which leads to

\begin{align}
R(Y,u)(\nabla_X Z)_{||} &= -\varepsilon g(\nabla_X u, Z)R(Y,u)u = -\varepsilon fg(\nabla_X u,Z)Y, \label{eqn:C-15}\\
R(Y,u)(\nabla_X Z)_\perp &= -\varepsilon f g(\nabla_X Z,Y)u = \varepsilon f g(Z,\nabla_X Y)u. \label{eqn:C-16}
\end{align}

\noindent 
The last relation follows from the fact that $g(Z,Y)$ is a constant along the
curve on which these vector fields are defined. Summing \eqref{eqn:C-15} and
\eqref{eqn:C-16} yields

\begin{equation}
R(Y,u)\nabla_X Z = -\varepsilon f \big( g(\nabla_X u,Z)Y - g(Z,\nabla_X Y)u \big)\label{eqn:C-17}
\end{equation}

\noindent 
and finally, substituting \eqref{eqn:C-11}, \eqref{eqn:C-12}, \eqref{eqn:C-13}
and \eqref{eqn:C-17} gives us

\begin{align}
(\nabla_X R)(Y,u)Z &=		\begin{aligned}[t]
								& \varepsilon df(X)g(Y,Z)u - \varepsilon fg(Y,Z)\nabla_X u\\
								& + \varepsilon f g(\nabla_X Y,Z)u\\
								& - h\big(g(\nabla_X u,Z)Y - g(Y,Z)\nabla_X u\big)\\
								& + \varepsilon f \big( g(\nabla_X u,Z)Y - g(Z,\nabla_X Y)u \big)
							\end{aligned}\nonumber\\
&=
							\begin{aligned}[t]
								& (h-\varepsilon f)\big(g(Y,Z)\nabla_X u - g(\nabla_X u,Z)Y \big)\\
								& -\varepsilon df(X)g(Y,Z)u
							\end{aligned}\label{eqn:C-18}
\end{align}

\noindent 
and


\begin{equation}
(\nabla_Y R)(u,X)Z = 	\begin{aligned}[t]
							& (h-\varepsilon f)\big(g(\nabla_Y u,Z)X - g(X,Z)\nabla_Y u\big) \\
							& + \varepsilon df(Y)g(X,Z)u.
						\end{aligned}\label{eqn:C-19}
\end{equation}

\noindent 
Substituting \eqref{eqn:C-9}, \eqref{eqn:C-18} and \eqref{eqn:C-19} leads to the result
\begin{equation}
0 =	\begin{aligned}[t]
		& (h-\varepsilon f)\Big[\big( g(Y,Z)\nabla_u X - g(X,Z)\nabla_u Y\big)\\
		& + \big(g(\nabla_Y u,Z)X - g(\nabla_X u,Z)Y \big)\\
		& + \big(g(Y,Z)\nabla_X u - g(X,Z) \nabla_Y u \big)\Big]\\
		& + dh(u)\big( g(Y,Z)X - g(X,Z)Y \big)\\
		& + \varepsilon\big( df(Y)g(X,Z) - df(X)g(Y,Z) \big)u.
	\end{aligned}\label{eqn:C-20}
\end{equation}

For $X\perp Y$ and $Z=Y$, the component of \eqref{eqn:C-20} which is parallel to 
$u$ gives us \eqref{eqn:3-1}. Eq. \eqref{eqn:3-2} follows by taking the normal 
component of \eqref{eqn:C-20} with respect to $u$, with the vector fields $X,Y,Z$
being orthogonal to $u$ and to each other. To prove \eqref{eqn:3-3} we will 
again use the second Bianchi identity,
\begin{equation}
(\nabla_X R)(Y,Z)Z + (\nabla_Y R)(Z,X)Z + (\nabla_Z R)(X,Y)Z = 0.\label{eqn:C-21}
\end{equation}

\noindent 
with the vector fields $X,Y,Z$ being orthogonal to $u$ and to each other. Then:
\begin{align}
(\nabla_X R)(Y,Z)Z 					&=	\begin{aligned}[t]
											& \nabla_X\big(R(Y,Z)Z\big) - R(\nabla_X Y,Z)Z\\
											& - R(Y,\nabla_X Z)Z - R(Y,Z)\nabla_X Z,
										\end{aligned}\label{eqn:C-22}\\
\nonumber\\
\nabla_X\big(R(Y,Z)Z\big)			&=	\begin{aligned}[t]
											& dh(X)\big(g(Z,Z)Y - g(Y,Z)Z\big) \\
											& + h\big(g(Z,Z)\nabla_X Y - g(Y,Z)\nabla_X Z\big)
										\end{aligned}\label{eqn:C-23}\\
									&=	g(Z,Z)\big(dh(X)Y + h\nabla_X Y\big),\nonumber
\nonumber\\
R\big((\nabla_X Y)_\perp,Z\big)Z	&=	h\big(g(Z,Z)(\nabla_X Y)_\perp - g(\nabla_X Y,Z)Z\big),\label{eqn:C-24}\\
\nonumber\\
R\big(Y,(\nabla_X Z)_\perp\big)Z	&=	h\big(g(\nabla_X Z,Z)Y - g(Y,Z)(\nabla_X Z)_\perp\big)=0,\label{eqn:C-25}\\
\nonumber\\
R(Y,Z)(\nabla_X Z)_\perp	&= h\big( g(Z,\nabla_X Z)Y - g(Y,\nabla_X Z)Z\big)\label{eqn:C-26}\\
							&= -hg(Y,\nabla_X Z)Z.\nonumber
\end{align}

\noindent 
Summing  \eqref{eqn:C-24} and \eqref{eqn:C-26} we get
\begin{align}
R\big((\nabla_X Z)_\perp, Z\big)Z + R(Y,Z)(\nabla_X Z)_\perp	&=	\begin{aligned}[t]
																		& hg(Z,Z)(\nabla_X Z)_\perp\\
																		& - \big(g(\nabla_X Y, Z) + g(Y,\nabla_X Z)\big)hZ
																	\end{aligned}\label{eqn:C-27}\\
																&= hg(Z,Z)(\nabla_X Z)_\perp,\nonumber
\end{align}

\noindent 
while from \eqref{eqn:C-23} and \eqref{eqn:C-27} we get 
\begin{equation}
\big[(\nabla_X R)(Y,Z)Z\big]_\perp = g(Z,Z)dh(X)Y,\label{eqn:C-28}
\end{equation}

\noindent 
and since
\begin{equation}
(\nabla_Y R)(Z,X)Z = -(\nabla_Y R)(X,Z)Z,\label{eqn:C-29}
\end{equation}

\noindent 
it follows that
\begin{equation}
\big[(\nabla_X R)(Y,Z)Z + (\nabla_Y R)(Z,X)Z \big]_\perp = g(Z,Z)\big(dh(X)Y - dh(Y)X).\label{eqn:C-30}
\end{equation}

As for the last term, 
\begin{equation}
(\nabla_Z R)(X,Y)Z =	\begin{aligned}[t]
							& \nabla_Z \big(R(X,Y)Z\big) - R(\nabla_Z X,Y)Z\\
							& - R(X,\nabla_Z Y)Z - R(X,Y)\nabla_Z Z.
						\end{aligned}\label{eqn:C-31}
\end{equation}

\noindent 
We have
\begin{align}
R\big((\nabla_Z X)_\perp,Y\big)Z	&=	\begin{aligned}[t]
											& h\big(g(Y,Z)(\nabla_Z X)_\perp - g(\nabla_Z X,Z)Y\big)\\
											& - hg(\nabla_Z X,Z)Y,
										\end{aligned}\label{eqn:C-32}\\
R\big(X,(\nabla_Z Y)_\perp\big)Z	&=	h\big(g(\nabla_Z Y,Z)X - g(X,Z)(\nabla_Z Y)_\perp\big)\label{eqn:C-33}\\
									&= hg(\nabla_Z Y,Z)X.\nonumber				
\end{align}

\noindent 
Summing these expressions,
\begin{equation}
R\big((\nabla_Z X)_\perp,Y\big)Z + R\big(X,(\nabla_Z Y)_\perp\big)Z + R(X,Y)(\nabla_Z Z)_\perp = 0,\label{eqn:C-34}\\
\end{equation}

\noindent 
and as such the second Bianchi identity \eqref{eqn:C-21} is equivalent to
\begin{equation}
g(Z,Z)\big(dh(X)Y - dh(Y)X\big) = 0,
\label{eqn:C-35}
\end{equation}

\noindent 
which concludes the proof.
\end{proof}

Eq. \eqref{eqn:C-20} can also be used to derive a relation involving $dh(u)$. As we will see in the next section, this relation will give us the second fundamental form of the hypersurfaces of constant time. 

\begin{theorem}
If $X,Y$ are vector fields that are orthogonal to $u$, then 
\begin{equation}
    g(X,\nabla_Y u) = -\frac{dh(u)}{2(h-\varepsilon f)}\, g(X,Y).
\end{equation}
\end{theorem}

\begin{proof}
As before, because this is a tensor equation, it suffices to prove it in an orthonormal basis of vector fields that are Fermi-parallel along the flow of $u$. Let $X,Y$ be non-null vector fields of this basis that are orthogonal to $u$. Assuming that $X,Y$ are orthogonal to each other and to
$u$, \eqref{eqn:C-20} gives us for $Y=Z$

\begin{equation}
dh(u) = -(h-\varepsilon f)\left(\frac{g(X,\nabla_X u)}{g(X,X)} + \frac{g(Y,\nabla_Y u)}{g(Y,Y)}\right).
\label{eqn:3-12}
\end{equation}

\noindent 
Since the dimension of $M$ is $\geq 4$ the aforementioned basis contains a vector field $Z$ orthogonal to $X,Y,u$ such that

\begin{equation}
dh(u) = -(h-\varepsilon f)\left(\frac{g(Z,\nabla_Z u)}{g(Z,Z)} + \frac{g(Y,\nabla_Y u)}{g(Y,Y)}\right)
\label{eqn:3-13}
\end{equation}

\noindent 
and as such 

\begin{equation}
\frac{g(X,\nabla_X u)}{g(X,X)} = \frac{g(Z,\nabla_Z u)}{g(Z,Z)}.\label{eqn:3-14}
\end{equation}

\noindent 
Thus 
\begin{equation}
\frac{g(X,\nabla_X u)}{g(X,X)} = \frac{g(Y,\nabla_Y u)}{g(Y,Y)}
\end{equation}
for every pair of non-null vector fields $X,Y$ orthogonal to each other and to $u$, which in turn leads to \begin{equation}
dh(u) = -2(h-\varepsilon f)\frac{g(X,\nabla_X u)}{g(X,X)}.\label{eqn:3-15}
\end{equation}
\noindent 
The proof then follows from the polarization identity and \eqref{eqn:3-2}.
\end{proof}

\section{Proof of Theorem~\ref{thm-main}}
\label{sec.Proof}

We can now prove the main result of this paper as stated in Section \ref{sec.Statement}. We start by showing that $u$ is orthogonal to a 1-parameter family of hypersurfaces. 

\begin{lemma}

Let $u^\flat = g(u,\cdot)$ be the metrically equivalent 1-form of the vector field $u$. Then 

\begin{equation}
d\left[(h-\varepsilon f)u^\flat \right] = 0.\label{eqn:3-4} 
\end{equation}
\end{lemma}

\begin{proof}
Let $\omega = (h-\varepsilon f) u^\flat$. Then for every pair of vector fields $X,Y$ on $M$
\begin{align}
d\omega(X,Y)	&=	X\left(\omega(Y)\right) - Y\left(\omega(X)\right) - \omega\left([X,Y]\right) \label{eqn:3-5} \\
				&=	\begin{aligned}[t]
						& g\left(\nabla_X\left[\left(h-\varepsilon f\right)u\right],Y\right) + g\left(\left(h-\varepsilon f\right)u,\nabla_X Y\right) \\
						& - g\left(\nabla_Y\left[\left(h-\varepsilon f\right)u\right],X\right) - g\left(\left(h-\varepsilon f\right)u,\nabla_Y X\right) \\
						& - g\left(\left(h-\varepsilon f\right)u,\nabla_X Y - \nabla_Y X \right)
					\end{aligned}\nonumber\\
				&= 	g\left(\nabla_X\left[\left(h-\varepsilon f\right)u\right],Y\right) - g\left(\nabla_Y\left[\left(h-\varepsilon f\right)u\right],X\right).\nonumber				
\end{align}
\noindent
The proof is thus concluded if and only if
\begin{equation}
g\left(\nabla_X\left[\left(h-\varepsilon f\right)u\right],Y\right) = 
                g\left(\nabla_Y\left[\left(h-\varepsilon f\right)u\right],X\right).
\label{eqn:3-6}
\end{equation}
If $X,Y$ are orthogonal to $u$ then due to \eqref{eqn:3-2}
\begin{align}
g\left(\nabla_X\left[\left(h-\varepsilon f\right)u\right],Y\right) &= (h-\varepsilon f) g(\nabla_X u, Y) \label{eqn:3-7} \\
&= (h-\varepsilon f) g(\nabla_Y u, X) \nonumber \\
&= g\left(\nabla_Y\left[\left(h-\varepsilon f\right)u\right],X\right). \nonumber 
\end{align}
\noindent
To finish proving this lemma, we need only prove the special case where $Y=u$ and $X$ is orthogonal to $u$. Using Eq.~(3.1) we get 
\begin{align}
g\left(\nabla_X\left[\left(h-\varepsilon f\right)u\right],u\right) &= -df(X)                    \nonumber \\
        &= \left(h-\varepsilon f\right) g\left(X,\nabla_u u\right)                              \nonumber \\
        &= g\left(X,\nabla_u\left[\left(h-\varepsilon f\right)u\right] \right) \label{eqn:3-8} 
\end{align}
\end{proof}

\noindent 
From Poincar\'e's lemma we have that for a sufficiently small neighborhood
$\mathscr{U}$ of $M$ there exists a function $t$ on $\mathscr{U}$ such that
\begin{equation}
dt = \left(h-\varepsilon f\right) u^\flat.\label{eqn:3-9}
\end{equation}
\noindent 
The domain $\mathscr{U}$ can be extended to all of $M$ if it is simply-connected.

\vspace{0.5em}

Since $dt\neq 0$ everywhere, the set $\Sigma_\tau$ of points $p\in\mathscr{U}$ that satisfies the equation $t(p)=\tau$ is a hypersurface of $\mathscr{U}$, while the one-parameter group of diffeomorphisms of $\partial_t$ maps the $\Sigma_\tau$ to $\Sigma_s$ for every pair of $\tau,s$. In order to find out whether the sectional curvature of the 1-parameter hypersurfaces is constant or not, we need to calculate their second fundamental form
\begin{equation}
\text{II}_\tau(X,Y) = \varepsilon \, g_\tau(\nabla_X Y,u)u 
= -\varepsilon \, g_\tau (X,\nabla_Y u)u
= \frac{\varepsilon dh(u)}{2(h-\varepsilon f)} \, g_\tau(X,Y)u,
\label{eqn:3-10}
\end{equation}
\noindent 
where $g_\tau$ is the induced metric on $\Sigma_\tau$. The substitution of the second fundamental form to the Gauss formula for the Riemann tensor $R_\tau$ for the hypersurface $\Sigma_\tau$
\begin{align}
g_\tau\big(R(X,Y)Z,W\big) = g_\tau\big(R_\tau(X,Y)Z,W\big) &+ g_\tau\big(\text{II}_\tau(X,Z),\text{II}_\tau(Y,W)\big)\label{eqn:3-16}\\
& - g_\tau\big(\text{II}_\tau(Y,Z),\text{II}_\tau(X,W)\big)\nonumber
\end{align}
\noindent 
gives us
\begin{equation}
R_\tau(X,Y)Z = \left(h + \varepsilon \left[\frac{dh(u)}{2(h-\varepsilon f)}\right]^2 \right)
              \left(g_\tau(Y,Z)X - g_\tau(X,Z)Y\right).
\label{eqn:3-17}
\end{equation}
From Schur's lemma, which holds for semi-Riemannian manifolds \cite{o1983semi},
we conclude that $\Sigma_\tau$ is a space of constant curvature for every
$\tau$. Therefore, the function 
\begin{equation}
K_\tau = h + \varepsilon \left[\frac{dh(u)}{2(h-\varepsilon f)}\right]^2,
\label{eqn:3-18}
\end{equation}
depends only on $\tau$.

\begin{proof}[Proof of Theorem~\ref{thm-main}]

\noindent 
For an arbitrary $\tau$ and $p\in\Sigma_\tau$ we choose the vectors
$e_1,...,e_{n-1}\in T_p\Sigma_\tau$ and extend them to vector fields
$E_1,...,E_{n-1}$ along the flow of $\partial_t$ using Lie transport:
\begin{equation}
[\partial_t,E_i] = 0, \;\; i=1,...,n-1.\label{eqn:3-19}
\end{equation}

\noindent 
The vectors $e_1,...,e_{n-1}$ are by definition orthogonal to $u$, implying that $\partial_t\perp E_i$ for every $i=1,...,n-1$. 

\vspace{0.5em}

Let us denote by $\pi$ the projection tensor to the orthogonal complement of
$u$. Using eq. \eqref{eqn:3-9}, $\pi$ can be expressed as
\begin{equation}
\pi = g -\varepsilon u^\flat \otimes u^\flat = g - \frac{\varepsilon}{(h-\varepsilon f)^2}dt \otimes dt\label{eqn:3-20},
\end{equation}
\noindent 
and its components with respect to the basis $\{E_i\}$ are
\begin{equation}
\pi_{ij} = g(E_i,E_j)\equiv g_{ij}.\label{eqn:3-21}
\end{equation}
We have
\begin{align}
g(E_i,\nabla_{E_j}u)	&= g\left(E_i,\nabla_{E_j}\left(\frac{1}{h-\varepsilon f}\partial_t\right)\right) \label{eqn:3-22}\\
						&= \frac{1}{h-\varepsilon f}\,g(E_i,\nabla_{E_j}\partial_t) \nonumber\\
						&= \frac{1}{h-\varepsilon f}\,g(E_i,\Gamma_{jt}^t\partial_t + \Gamma_{jt}^k E_k)\nonumber\\
						&= \frac{1}{h-\varepsilon f}\,g_{ik}\Gamma_{jt}^k\nonumber\\
						&= \frac{1}{h-\varepsilon f}\,g_{ik}\cdot\frac{1}{2}g^{k\mu}(g_{\mu j,t} + g_{\mu t,j} - g_{jt,\mu})\nonumber\\
						&= \frac{1}{2(h-\varepsilon f)}\,\delta_i^\mu(g_{\mu j,t} + g_{\mu t,j})\nonumber\\
						&= \frac{1}{2(h-\varepsilon f)}\,g_{ij,t},\nonumber 
\end{align}
\noindent 
On the other hand,
\begin{equation}
dh(u) = \frac{1}{h-\varepsilon f}\,\partial_t h,\label{eqn:3-23}
\end{equation}
\noindent 
and therefore
\begin{equation}
\frac{dh(u)}{h-\varepsilon f} = \frac{1}{(h-\varepsilon f)^2}\,\partial_t h.\label{eqn:3-24}
\end{equation}
\noindent 
Since $dh(X)=0$ for every vector field $X$ orthogonal to $u$, $h$ is a function of $\tau$ and the left term of \eqref{eqn:3-24} is also only a function of $\tau$. Furthermore, the right term of \eqref{eqn:3-24} is a product of a function of $\tau$ with $K_\tau$. Thus, it follows that $f$ only depends on $\tau$:
\begin{equation}
df(X) = 0, ~\textrm{for~every}~ X\perp u,\label{eqn:3-25}
\end{equation}
\noindent 
which also shows that the flow of $u$ is geodesic (cf. \eqref{eqn:3-1}). Lastly, using \eqref{eqn:3-25} we get
\begin{equation}
g_{ij,\,t} = -\frac{\partial_t h}{h-\varepsilon f}\,g_{ij},\label{eqn:3-26}
\end{equation}
\noindent 
or, equivalently, 
\begin{equation}
\mathcal{L}_{\partial_t}\pi = -\frac{\partial_t h}{h-\varepsilon f}\pi.\label{eqn:3-27}
\end{equation}
\noindent 
Eq. \eqref{eqn:3-26} can be integrated in a coordinate independent way. Let us choose an arbitrary point $p\in\Sigma_0$ and arbitrary vectors $x,y\in T_pM$. The function
\begin{equation}
\psi = -\frac{\partial_t h}{h-\varepsilon f}\label{eqn:3-28}
\end{equation}
\noindent 
is constant on the hypersurfaces $\Sigma_\tau$, and the function $\tau\longrightarrow F(\tau)$ defined by
\begin{equation}
F(\tau) = (\phi_\tau^*\pi)(x,y),\label{eqn:3-29}
\end{equation}
\noindent 
where $\{\phi_\tau\}$ is the one-parameter family of diffeomorphisms of
$\partial_t$, is smooth and its derivative is equal to
\begin{align}
F'(\tau)	&= \frac{d}{ds}F(\tau + s)\bigg\rvert_{s=0} \nonumber \\
			&= \frac{d}{ds}\bigg\rvert_{s=0}\left(\phi_\tau^*  \phi_s^* \pi\right)(x,y)\nonumber\\
			&= (\phi_\tau^* \mathcal{L}_{\partial / \partial t}\pi)(x,y)\nonumber\\
			&= (\phi_\tau^* \psi\pi)(x,y)\nonumber\\
			&= \psi\circ\phi_\tau(p) F(\tau). \label{eqn:3-30}
\end{align}
\noindent 
The proof is concluded by solving \eqref{eqn:3-30} and then changing variables
\begin{equation}
\bigg\rvert\frac{1}{h-\varepsilon f}\bigg\rvert dt \longrightarrow dt.\label{eqn:3-31}
\end{equation}
\end{proof}

\section{Conclusions}
\label{sec.Impl}

In this work we derive a geometrical necessary and sufficient condition for a
semi-Riemannian space to be locally isotropic with respect to a unit vector
field $u$. We prove that this is fulfilled when the vector field pointwise
distinguishes two sectional curvatures. Our result might be envisaged as an
extension of the theorem concerning local isometry between spaces of equal
constant curvature to a class of spaces of non-constant curvature. Due to the
properties of the Riemann tensor, this condition can be expressed in terms of
it without any further assumption. 
Since we do not need to make any kind of assumption about the field equations
or matter fields, this represents an advantage from the point of view of
theoretical cosmology and the study of alternative theories of gravity.

\vspace{0.5em}

In addition, the local isometry is extended to a global one if the space is
simply-connected.  Since the equation of geodesic deviation for a specific
class of geodesics parallel or orthogonal to $u$ is independent of direction,
our condition is an \textit{isotropy of geodesic deviation}.

\section{Acknowledgements}

This work was supported by the National Key R\&D Program of China (2016YFA0400702) and the National Science Foundation of China (11721303, 11873022 and 11991053). We thank H. Ringstr\"om and L. Andersson for their input.

\appendix

\section*{Appendix: Fermi-parallel frame fields}
\setcounter{section}{1}

A \textit{Fermi-parallel} orthonormal basis of vector fields along a timelike curve 
$\gamma$ of unit speed $u$ is a set of vector fields $E_1,...,E_n$
such that $E_n=u$ and
\begin{equation}
\nabla_u E_i = g(\nabla_u u, E_i)u, \;\; i=1,...,n-1.\label{eqn:B-1}
\end{equation}
\noindent 
In our work we use these vector fields to simplify calculations involving the
Riemann and Ricci tensors. To define such a basis along an arbitrary non-null
curve we need a generalization of the Fermi derivative along non-null vector
fields.

\vspace{0.5em}

Following the work of \cite{hawking1973large}, we seek a derivation
$\mathcal{D}_u$, with $u$ a non-null vector field that satisfies the
following two conditions:
\begin{enumerate}
	\item $\mathcal{D}_u u = 0$,
	\item If $X$ is a vector field orthogonal to $u$ then $\mathcal{D}_u X = (\nabla_u X)_\perp$.
\end{enumerate}
\noindent 
From the first property we have that
\begin{align}
\mathcal{D}_u X_{||} &= \mathcal{D}_u \left(\varepsilon g(u,X)u \right) \nonumber \\
&= \mathcal{D}_u\left(\varepsilon g(u,X)\right)u \nonumber\\
&= \nabla_u \left(\varepsilon g(u,X)\right)u \nonumber\\
&= \nabla_u X_{||} - \varepsilon g(X,u)\nabla_u u, \label{eqn:B-2}
\end{align}
\noindent 
while from the second property we have
\begin{align}
\mathcal{D}_u X_\perp &= \left(\nabla_u X_\perp \right)_\perp \label{eqn:B-3}\\
&= \nabla_u X_\perp - \varepsilon g(\nabla_u X_\perp, u)u \nonumber\\
&= \nabla_u X_\perp - \varepsilon \left[ \nabla_u \left(g\left(X_\perp,u\right)\right) - g\left(X_\perp,\nabla_u u\right) \right]u\nonumber\\
&= \nabla_u X_\perp + \varepsilon g\left(X,\nabla_u u\right)u \nonumber
\end{align}
\noindent 
The sum of Eqs.~\eqref{eqn:B-2} and \eqref{eqn:B-3} gives us the generalization of
the Fermi derivative for unit vector fields,
\begin{equation}
\mathcal{D}_u X = \nabla_u X + \varepsilon g\left(X,\nabla_u u\right)u - \varepsilon g(X,u)\nabla_u u.\label{eqn:B-4}
\end{equation}

\begin{lemma}
Let $\gamma$ be an integral curve of a unit vector field $u$. The product $g(X,Y)$ of 
two vector fields $X,Y$ on $\gamma$ that are \textit{Femi-parallel} along $\gamma$:
\begin{equation}
\mathcal{D}_u X = \mathcal{D}_u Y = 0 \label{eqn:B-5}
\end{equation}
\noindent 
is constant along $\gamma$. 
\end{lemma}

\begin{proof}
Simple substitution yields
\begin{equation}
\frac{d}{d\tau} g(X,Y) = g\left(\nabla_u X,Y\right) + g\left(X,\nabla_u Y\right) = 
                         g\left(\mathcal{D}_u X,Y\right) + g\left(X,\mathcal{D}_u Y\right).
\label{eqn:B-6}
\end{equation}
\end{proof}

An immediate consequence of this Lemma is that given any curve $\gamma$ of unit speed $u$, 
a vector field $X$ which is Fermi-parallel along $\gamma$ and orthogonal to $u$ at 
some point must be everywhere orthogonal to $u$. Therefore its covariant derivative 
along $\gamma$ is equal to 
\begin{equation}
\nabla_u X = - \varepsilon g(X,\nabla_u u)u. 
\label{eqn:B-7}
\end{equation}
\noindent 
Thus, we can always construct a \textit{Fermi-parallel} set of vector fields
$E_1,...,E_n$ along $\gamma$ such that
\begin{enumerate}
	\item $E_n\equiv u$,
	\item $\nabla_u E_i = -\varepsilon g\left(E_i,\nabla_u u\right) u, \;\; i=1,...,n-1$,
	\item At each point $p$ of $\gamma$ they form an orthonormal basis of $T_pM$.
\end{enumerate}

\end{document}